\documentclass[12pt,reqno]{amsart}
\usepackage{amsmath,xcolor,cases}
\usepackage{fullpage}

\theoremstyle{plain}
\newtheorem{thm}{Theorem}[section]
\newtheorem{lem}[thm]{Lemma}
\theoremstyle{definition}
\newtheorem{defn}[thm]{Definition}

\newtheorem{rem}[thm]{Remark}
\newtheorem{cor}[thm]{Corollary}
\numberwithin{equation}{section}
\newcommand{\R}{\mathbb{R}}
\renewcommand{\a}{\alpha}
\renewcommand{\c}{\gamma}
\newcommand{\e}{\varepsilon}
\renewcommand{\d}{\delta}
\newcommand{\F}{\mathcal{F}}
\usepackage{color}

\newenvironment{dedication}{\begin{quotation}\small\begin{em}}   {\par\end{em}\end{quotation}\vspace{1em}}
        
\begin{document}
\title[BVPs for functional differential equations]{Nontrivial solutions of boundary value problems for second order functional differential equations}
\date{}

% ----------------------------------------------------------------


\author[A. Calamai]{Alessandro Calamai}
\address{Alessandro Calamai, 
Dipartimento di Ingegneria Industriale e Scienze Matematiche,
Universit\`{a} Politecnica delle Marche
Via Brecce Bianche
I-60131 Ancona, Italy}%
\email{calamai@dipmat.univpm.it}%

\author[G. Infante]{Gennaro Infante}
\address{Gennaro Infante, Dipartimento di Matematica e Informatica, Universit\`{a} della
Calabria, 87036 Arcavacata di Rende, Cosenza, Italy}%
\email{gennaro.infante@unical.it}%

\begin{abstract}
In this paper we present a theory for the existence of multiple nontrivial solutions for a class of perturbed Hammerstein integral equations. Our methodology, rather than to work directly in cones, is to utilize the theory of fixed point index on affine cones.  This approach is fairly general and covers a class of nonlocal boundary value problems for functional differential equations. Some examples are given in order to illustrate our theoretical results.
\end{abstract}

\maketitle

\begin{dedication}
\begin{center}
Dedicated to Massimo Furi, Professor Emeritus at the University of Florence
\end{center}
\end{dedication}

\section{Introduction}
The problem of the existence of solutions for functional differential equations (FDEs) has been discussed by a large number of researchers. A survey of classical and recent results in this topic goes beyond the scopes of this manuscript; we refer the reader to the books by Hale and Lunel~\cite{halelunel}, Erbe and co-authors~\cite{ekz-book}, 
Agarwal et al.~\cite{abbd-book}, the survey by Ntouyas~\cite{sotiris-fde}, the papers by Nussbaum~\cite{Nuss}, Xu and Liz~\cite{xuliz}, Ntouyas, Sficas and Tsamatos~\cite{nst93}, and references therein. The motivation for these studies, apart from a purely mathematical interest, relies in the fact that these type of equations arise quite frequently when modelling physical problems, see for example the ones illustrated in~\cite{halelunel}. Regarding  the existence of positive solutions,
fixed point techniques  in cones have been used, for example,
by Wang~\cite{wang04} (in the first order case) and by Erbe and Kong~\cite{ek94},
Karakostas, Mavridis and Tsamatos~\cite{kmt03} and Ma~\cite{ma07} in the second order case
under \emph{local} boundary conditions. The existence of positive solutions in the \emph{nonlocal} case has been studied by Karakostas, Mavridis and Tsamatos~\cite{kmt04}
and, more recently, by Karaca~\cite{karaca13}. 

In particular, in the paper~\cite{kmt04}, the authors study the existence of positive solutions of the functional boundary value problem (FBVP)
\begin{equation}\label{kmt-eq}
u''(t)+F(t,u_t)=0,\ t \in [0,1],
\end{equation}
with initial conditions
\begin{equation}\label{kmt-i}
u(t)=\psi(t),\ t \in [-r,0],
\end{equation}
and boundary conditions (BCs)
\begin{equation}\label{kmt-bcs}
u(0)=0,\;  u(1)=\int_{t_1}^{t_2}u(s)dA(s),\; {t_1,\,t_2}\in (0,1),
\end{equation}
where $\psi$ is assumed to be non-negative and the above integral is meant in the Riemann-Stieltjes sense and is given by a \emph{positive} measure.

The methodology in~\cite{kmt04} is to re-write the FBVP~\eqref{kmt-eq}-\eqref{kmt-i}-\eqref{kmt-bcs} as an Hammerstein-type integral equation of the form
$$
u(t)=\int_{0}^{1} \hat{k}(t,s)F(s,u_s)\,ds,
$$
with a suitable kernel $\hat{k}$ and to use the Leggett-Williams theorem~\cite{lw79}.

Our approach is somewhat different and we study, in the spirit of the paper by  Infante and Webb in~\cite{gijwems}, the existence of nontrivial solutions
of perturbed Hammerstein integral equations of the type 
\begin{equation}\label{eqhamm-intro}
u(t)=\psi (t) + \int_{0}^{1} k(t,s)g(s)F(s,u_s)\,ds + \gamma(t) \alpha[u],
\end{equation}
where $\alpha[\cdot]$ is a linear functional given by a Stieltjes integral, namely
\begin{equation}\label{nbc}
\alpha[u]=\int_0^1 u(s)\,dA(s).
\end{equation}
Here by a \emph{nontrivial solution} of \eqref{eqhamm-intro} we mean a solution that does not coincide with $\psi$; furthermore we stress that the solutions that we obtain are positive on a subinterval $[a,b]$ of $[0,1]$ and are allowed to \emph{change sign} in $[0,1]$.

We point out that the formulation \eqref{nbc} involves a \emph{signed} measure and covers the case of multi-point and integral conditions, namely
$$
\alpha[u]=\sum_{j=1}^m\a_j u(\eta_j)\quad\text{or}\quad \alpha[u]=\int_{0}^1\varphi(s)u(s)ds.
$$
Multi-point and integral BCs are widely studied objects in the case of ODEs. As far as we know multi-point BCs were investigated for the first time in 1908 by Picone~\cite{Picone}. In 1942 Whyburn~\cite{Whyburn} wrote a review on differential equations with general BCs that included also integral BCs involving Stieltjes measures. We mention also the (more recent) reviews of Conti~\cite{Conti}, Ma~\cite{rma}, Ntouyas~\cite{sotiris} and \v{S}tikonas~\cite{Stik} and the papers by Karakostas and Tsamatos~\cite{kttmna, ktejde} and by Webb and Infante~\cite{jwgi-lms}.

One advantage of studying the solutions of the perturbed integral equation~\eqref{eqhamm-intro} is that it provides a fairly general setting that covers, \emph{as special cases}, a number of FBVPs subject to nonlocal conditions. 

A new feature of the present paper is that we work in \emph{affine cones}.
In fact, due to the presence of the delay and of the initial datum $\psi$, in order
to investigate the solutions of the integral equation~\eqref{eqhamm-intro}
we find it convenient and natural to work in translates of cones in Banach spaces,
rather than to work directly in cones.
In order to do this, we provide a modification, tailored for our setting,
of some classical result on the fixed point index.

As already pointed out, one benefit
of providing an existence theory for the perturbed integral equation~\eqref{eqhamm-intro} is that it gives results for a (relatively) large class of FBVPs. As an example, we illustrate here the applicability of our results to the nonlocal FBVP
\begin{equation}\label{eq-ther}
-u''(t)=g(t) F(t,u_t),\ t \in [0,1],
\end{equation}
with initial conditions
\begin{equation}\label{eqic-ther}
u(t)=\psi(t),\ t \in [-r,0]
\end{equation}
and BCs
\begin{equation}\label{eqbc-ther}
u(0)=0,\; \beta u'(1) + u(\eta)=\a[u],\; {\beta}>0,\; {\eta}\in (0,1).
\end{equation}
The FBVP~\eqref{eq-ther}-\eqref{eqic-ther}-\eqref{eqbc-ther} can be seen as a retarded analogous of some thermostat problems with nonlocal controllers studied, in the case of ODEs, by Infante and Webb~\cite{gijwnodea, gijwems}, who were motivated by earlier work of Guidottti and Merino~\cite{guimer}.
Thermostat problems of this type have been studied by a number of authors,
more information can be found in the recent papers \cite{gi-pp-ft, jw-narwa} and references therein.

We describe the applicability of our theory in the special case of a delay differential equation,
that is,
\[
-u''(t)=g(t) f(t,u(t),u(t-r)),\ t \in [0,1].
\]
As an application, we discuss the existence of nontrivial solutions of the 
delay equation
\begin{equation}\label{eqdelay1-intro}
-u''(t)= \lambda \, |u(t)|^{p-1} |u(t-r)|,\ t \in [0,1],
\end{equation}
where $p \geq 1$,
with the initial conditions \eqref{eqic-ther}
and some nonlocal BCs.

Finally, we illustrate how our approach can be applied to the case of positive solutions; this is done for the FBVP~\eqref{eq-ther}-\eqref{eqic-ther}-\eqref{eqbc-ther},
 and also for the FBVP~\eqref{kmt-eq}-\eqref{kmt-i}-\eqref{kmt-bcs},  complementing the results of~\cite{kmt04,jw-na05}.

\section{Fixed points on translates of a cone}\label{translates}
In this Section we provide some useful properties of the fixed point index on a translate of a cone $K$, in the spirit of Remark 1 of \cite{amannjfa}. These properties are used in Section~\ref{sec-changesign} to prove our existence and multiplicity results for the integral equation~\eqref{eqhamm-intro}.

Let $X$ be a Banach Space. A \emph{cone} on $X$ is a closed,
convex subset of $X$ such that $\lambda \, x\in K$ for $x \in K$ and
$\lambda\geq 0$ and $K\cap (-K)=\{0\}$.
If $\Omega$ is a bounded open subset of $K$ (in the relative
topology) we denote by $\overline{\Omega}$ and $\partial \Omega$
the closure and the boundary of $\Omega$ relative to $K$.
Given $y\in X$, we can consider the \emph{translate} of a cone $K$, namely
$$
K_y:=y+K=\{y+x: x\in K\}.
$$
When $D$ is an open
bounded subset of $X$ we write $D_{K_y}=D \cap K_y$, an open subset of $K_y$.

Observe that translates of cones are examples of
\emph{absolute neighborhood retracts} (ANRs).
So the classical fixed point index theory for
compact maps on cones (see e.g.\ \cite{amannjfa, amann,  guolak})
can be extended to the context of translates of cones. 
The fixed point index satisfies properties analogous to those of the
classical Leray-Schauder degree.
The reader can see for
instance \cite{Br, DuGr, Nu93} for a
comprehensive presentation of the index theory for ANRs.

The proof of the following Lemma can be carried out as in the case of cones, see for example the proof of Lemma 12.1
in the review \cite{amann}. We give here an explicit proof
for the sake of completeness;
for more details, see also the recent paper \cite{djeb2014}.

\begin{lem} \label{lemind}
Let $D$ be an open bounded set with $y \in D$.
Assume that $\F:\overline{D}_{K_y}\to K_y$ is
a compact map such that $x\neq \F x$ for $x\in \partial D_{K_y}$. Then
the fixed point index $i_{K_y}(\F, D_{K_y})$ has the following properties.
\begin{itemize}
\item[(1)] If there
exists $e\in K\setminus \{0\}$ such that $x\neq \F x+\sigma e$ for
all $x\in \partial D_{K_y}$ and all $\sigma >0$, then $i_{K_y}(\F, D_{K_y})=0$.
\item[(2)] If 
$\mu (x-y) \neq \F x-y$
for all $x\in
\partial D_{K_y}$ and for every $\mu \geq 1$, then $i_{K_y}(\F, D_{K_y})=1$.
\item[(3)] Let $D'$ be open in $X$ with
$\overline{D'}\subset D_{K_y}$. If $i_{K_y}(\F, D_{K_y})=1$ and
$i_{K_y}(\F, D'_{K_y})=0$, then $\F$ has a fixed point in
$D_{K_y}\setminus \overline{D'}_{K_y}$. The same result holds if
$i_{K_y}(\F, D_{K_y})=0$ and $i_{K_y}(\F, D'_{K_y})=1$.
\end{itemize}{}
\end{lem}
\begin{proof}
(1) Let $\alpha=\sup \{  \|x\| : x \in D_{K_y} \}$, $\beta=\sup \{  \|\F(x)\| : x \in D_{K_y} \}$ and let
$\gamma> \dfrac{\alpha+\beta}{\|e\|}$.
Define $H:[0,1] \times \overline{D}_{K_y} \to E$ by
$H(\lambda,x)= \F(x) + \lambda \gamma e$.
Note that $H$ is a compact map with values in $K_y$.
By the Homotopy invariance property, we get
$i_{K_y}(\F, D_{K_y})=i_{K_y}(\F+\gamma e , D_{K_y})$.

Assume now that $i_{K_y}(\F, D_{K_y}) \ne 0$.
Then, there exists $\bar x \in  D_{K_y}$ such that
$\bar x = \F(\bar x) + \gamma e$.
Consequently,
$\| \bar x \| \ge \gamma \|e\| - \| \F(\bar x)\| \ge \gamma \|e\| - \beta >\alpha$,
which is a contradiction. Hence, $i_{K_y}(\F, D_{K_y}) = 0$.

(2) Define $H:[0,1] \times \overline{D}_{K_y} \to E$ by
$H(\lambda,x)=(1-\lambda)y+\lambda \F(x)$.
Observe that $H$ is a compact map with values in $K_y$.
Thus, by the Homotopy invariance and Normalization properties, we have
$i_{K_y}(\F, D_{K_y})=i_{K_y}(y, D_{K_y})=1$

(3) This is a consequence of the Additivity and Solution properties.
\end{proof}

\section{Nontrivial solutions for a class of perturbed integral equations} \label{sec-changesign}

Given a compact interval $I\subset \R$, by $C(I, \R)$ we mean the Banach space of the
continuous functions defined on $I$ with the usual supremum norm.
Since we work with functions defined on different intervals (usually $I=[-r,0]$ or $I=[-r,1]$ with $r>0$), for sake of clarity
the norm of $u \in C(I, \R)$ will be denoted by $\|u\|_{I}$.

Given $r>0$ and a continuous function $u: J \to \R$, defined on a real interval
$J$, and given $t \in \R$ such that $[t-r, t] \subseteq J$, we adopt the
standard notation $u_t : [-r, 0] \to \R$ for the function defined by
$u_t (\theta) = u(t + \theta)$.

Let us consider the following integral equation in the space $C([-r, 1], \R)$:
\begin{equation}\label{eqhamm}
u(t)=\psi(t) + \int_{0}^{1} k(t,s)g(s)F(s,u_s)\,ds + \gamma(t) \alpha[u] =:\F u(t),
\end{equation}
where 
$$\alpha[u]=\int_{0}^{1} u(s)\,dA(s).$$

We require the following assumptions on the maps $F$, $k$, $\psi$, $\gamma$, $\alpha$ and $g$ that occur in \eqref{eqhamm} and on the delay $r$:
\begin{enumerate}
\item [$(C_{1})$] The function $\psi: [-r,1] \to \R$ is continuous and
such that $\psi(t)=0$ for all $t\in[0,1]$.
\item [$(C_{2})$] The kernel $k:[-r,1] \times [0,1] \to \R$ is measurable,
verifies $k(t,s)=0$ for all $t\in[-r,0]$ and almost every (a.\,e.) $s \in[0,1]$, and
for every $\bar t \in
[0,1]$ we have
\begin{equation*}
\lim_{t \to \bar t} |k(t,s)-k(\bar t,s)|=0 \;\text{ for a.\,e. } s \in
[0,1].
\end{equation*}{}
\item [$(C_{3})$]
 There exist a subinterval $[a,b] \subseteq (0,1]$, a measurable function
$\Phi$ with $\Phi \geq 0$ a.\,e., and a constant $c_1=c_1(a,b) \in (0,1]$ such that
\begin{align*}
|k(t,s)|\leq \Phi(s) \text{ for  all }  &t \in [0,1] \text{ and a.\,e. } \, s\in [0,1],\\
k(t,s) \geq c_1\,\Phi(s) \text{ for  all } &t\in [a,b] \text{ and a.\,e. } \, s \in [0,1].
\end{align*}{}
\item [ $(C_{4})$] The function $g:[0,1] \to \R$ is measurable, $g(t) \geq 0$ a.\,e. $t \in [0,1]$,
and satisfies that $g\,\Phi \in L^1[0,1]$
and $\int_a^b \Phi(s)g(s)\,ds >0$.{}
\item [$(C_{5})$] $F: [0,1] \times C([-r, 0], \R) \to [0,\infty)$ is an operator that satisfies some 
Carath\'eodory-type conditions (see also \cite{halelunel}); namely,
for each $\phi$, $t \mapsto F(t,\phi)$ is measurable and for a.\,e. $t$, $\phi \mapsto F(t,\phi)$ is continuous. Furthermore,
for each $R>0$, there exists $\varphi_{R} \in
L^{\infty}[0,1]$ such that{}
\begin{equation*}
F(t,\phi) \le \varphi_{R}(t) \ \text{for all} \ \phi \in C([-r, 0], \R)
\ \text{with} \ \|\phi\|_{[-r,0]} \le R,\ \text{and a.\,e.}\ t\in [0,1].
\end{equation*}{}
\item[$(C_{6})$]
$A$ is of bounded variation  ($\operatorname{Var}(A)<+\infty$), and
$\mathcal{K}_A(s):=\int_{{0}}^{{1}} k(t,s)dA(t) \geq 0$ for a.e. 
$s \in [0,1]$.
\item [ $(C_{7})$]
The function $\gamma: [-r,1] \to \R$ is continuous, $\gamma\not\equiv 0$ and
such that $\gamma(t)=0$ for all $t\in[-r,0]$; moreover,
$0 \leq \alpha[\gamma] <1\;
\text{and there exists}\; c_{2} \in(0,1] \;\text{such that}\;
\gamma(t) \geq c_{2}\|\gamma\|_{[0,1]} \;\text{for all }\; t \in [a,b]$.
\item [ $(C_{8})$] The inequality $r<b-a$ holds.
\end{enumerate}

We stress that, in particular, the assumption $(C_{5})$ is crucial to prove the compactness of the operator $\F$
(see Theorem \ref{thmk} below).

In the Banach space $C([-r, 1], \R)$ we define the cone
$$
K_0=\{u\in C([-r, 1], \R): u(t)=0\ \text{for all}\ t\in[-r,0],  \min_{t \in [a,b]}u(t)\geq c \|u\|_{[-r,1]}, \alpha[u] \ge 0\},
$$
where $c=\min\{c_1,c_2\}$. Note that $K_0 \neq \{0\}$ since $\gamma \in K_0$ and, furthermore, that the functions in $K_0$ are non-negative in the subset $[a,b]$ and are allowed to \emph{change sign} in $[0,1]$. The cone $K_0$ is a modification of the cone of functions introduced by Infante and Webb in~\cite{gijwjiea}. The idea of 
incorporating the functional $\alpha$ within the definition of the cone (this allows the use of signed measures) can be found in~\cite{jwgi-nodea-08} for the case of positive functions and in~\cite{Cab1} for the case of functions that are allowed to change sign.

We consider the following translate of the cone $K_0$,
$$K_\psi=\psi + K_0 = \{\psi +u : u \in K_0\}.$$
\begin{defn} \label{translcone}
We define the following subsets of $C([-r, 1], \R)$: $$K_{0,\rho}:=\{u\in K_0: \|u\|_{[0,1]} <\rho\},\ 
V_{0,\rho}:=\{u \in K_0: \displaystyle{\min_{t\in [a,b]}}u(t)<\rho \}$$
and the corresponding translates
$$K_{\psi,\rho}:= \psi + K_{0,\rho}, \  V_{\psi,\rho}:= \psi + V_{0,\rho}.$$
\end{defn}
Observe that
$\partial K_{\psi,\rho} = \psi + \partial K_{0,\rho}$ and $\partial V_{\psi,\rho} = \psi + \partial V_{0,\rho}$.
Let us stress that a key feature 
of these sets is that they can be
nested
$$
K_{\psi,\rho}\subset V_{\psi,\rho}\subset K_{\psi,\rho/c}.
$$
Furthermore, note that
$u\in  K_{\psi}$ means that $u=\psi+v$ with $v \in  K_{0}$ and, therefore, we have
\begin{equation}\label{norm-tra}
 \|u\|_{[-r,1]} =\max\{ \|\psi\|_{[-r,0]}, \|v\|_{[0,1]}\}.
\end{equation}
\begin{thm}\label{thmk}
Assume that the hypotheses $(C_{1})$-$(C_{8})$ hold for some $R>0$.
Then $\F$ maps
$\overline{K}_{\psi,R}$ into $K_\psi$ and is compact. When these hypotheses
hold for every $R>0$, $\F$ is compact and maps $K_\psi$ into $K_\psi$.
\end{thm}

\begin{proof}
Let $R>0$ be given and let $u \in \overline{K}_{\psi,R}$.
Let us show that $\F u -\psi \in K_0$.
First of all observe that our assumptions imply that $\F u$ is continuous on $[-r,1]$ and that
$\F u(t)-\psi(t)=0$ for $t \in [-r,0]$.
Now, for every $t \in [0,1]$ we have
\begin{align*}
|\F u(t) -\psi(t)| & \leq \int_{0}^{1}
|k(t,s)| g(s) F(s,u_s)\,ds + |\gamma(t)| \alpha[u] \\
&\leq    \int_{0}^{1} \Phi(s)g(s) F(s,u_s)\,ds +\|\gamma\|_{[0,1]} \alpha[u],
\end{align*}
moreover, since $[a,b]\subseteq (0,1]$,
$$
\min_{t\in [a,b]} \big(\F u(t) -\psi(t) \big) \geq c_1
\int_{0}^{1} \Phi(s)g(s) F(s,u_s)\,ds
+ c_2 \|\gamma\|_{[0,1]} \alpha[u]
\geq c\|\F u -\psi\|_{[-r,1]}.
$$
Furthermore, by $(C6)$, $\alpha$ is a bounded linear
operator. Using $(C6)$ and $(C7)$ we have
$$
\alpha[\F u]=\alpha[\gamma]\alpha[u]+
\int_{0}^{1} \mathcal{K}_{A}(s)g(s)F(s,u_s)\,ds\geq 0.
$$
Therefore we have that $\F u\in K_\psi$ for every $u\in
\overline{K}_{\psi,R}$.

To prove the compactness of $\F$,
let $\{u^n\}$ be a sequence in $C([-r, 1], \R)$ with $\|u^n\|_{[-r,1]}<R$.
Observe that $\|u_t^n\|_{[-r,0]}<\hat R$ for all $t \in [0,1]$, where $\hat R := R+\|\psi\|_{[-r,0]}$.
Consequently, for $t \in [-r,1]$ we have
\begin{align*}
|\F u^n(t)|& \leq |\psi(t)| + \int_{0}^{1}
|k(t,s)| g(s) F(s,u_s^n)\,ds
+ |\gamma(t)| \alpha[u^n]\\
&\leq \|\psi\|_{[-r,0]} + \int_{0}^{1} \Phi(s)g(s) \varphi_{\hat R}(s)\,ds
+\|\gamma\|_{[0,1]} \, R \operatorname{Var} (A).
\end{align*}
Hence, the sequence $\{\F u^n\}$ is bounded.

Now, by assumption $(C_{2})$ and Lebesgue's dominated convergence theorem, the function
$t\mapsto\int_{0}^{1}k(t,s) g(s) \varphi_{\hat R}(s)\,ds$
is continuous.
Let $\e >0$ be given; 
by the continuity of $\psi$ and $\gamma$ and of
$\displaystyle\int_{0}^{1}k(\cdot,s) g(s) \varphi_{\hat R}(s)\,ds$, 
there exists $\d>0$ such that:
\begin{itemize}
\item[] $|\psi(t_1)-\psi(t_2)|<\e$, provided that $t_1,t_2 \in [-r,0]$ with $|t_1-t_2|<\d$;
\item[] $|\gamma(t_1)-\gamma(t_2)|<\e$, provided that $t_1,t_2 \in [0,1]$ with $|t_1-t_2|<\d$;
\item[] $\displaystyle\int_{0}^{1}|k(t_1,s)-k(t_2,s)| g(s) \varphi_{\hat R}(s)\,ds <\e$, provided that $t_1,t_2 \in [0,1]$ with $|t_1-t_2|<\d$.
\end{itemize}
Therefore we have
$$
|\F u^n(t_1)-\F u^n(t_2)| \leq
|\psi(t_1)-\psi(t_2)| < \e,$$
if $t_1,t_2 \in [-r,0]$ with $|t_1-t_2|<\d$;
\begin{align*}
|\F u^n(t_1)-\F u^n(t_2)| \leq &
\int_{0}^{1}
|k(t_1,s)-k(t_2,s)| g(s) F(s,u_s^n)\,ds+|\gamma(t_1)-\gamma(t_2)| \alpha[u^n]\\
< &  \e\bigl(1+R \operatorname{Var}(A)\bigr),
\end{align*}
if $t_1,t_2 \in [0,1]$ with $|t_1-t_2|<\d$;
\begin{align*}
|\F u^n(t_1)-\F u^n(t_2)| \leq &
|\F u^n(t_1)-\F u^n(0)| +
|\F u^n(0)-\F u^n(t_2)| \\
< &\e\bigl(2+R \operatorname{Var}(A)\bigr),
\end{align*}
whenever $-r\leq t_1<0<t_2\leq 1$ with $|t_1-t_2|<\d$. 

Therefore the sequence $\{\F u^n\}$ is equicontinuous. The compactness of $\F$ now follows from the Ascoli-Arzel\`a Theorem.
\end{proof}

In the sequel, we give a condition that ensures that the index is 1 on $K_{\psi,\rho}$
for a suitable $\rho$ larger than the norm of $\psi$. 
\begin{lem}
\label{ind1b} Assume that 
\begin{enumerate}
\item[$(\mathrm{I}_{\protect\rho }^{1})$] \label{EqB} there exists $\rho> \|\psi\|_{[-r,0]}$ such that
$$
\frac{F^{(-\rho,\rho)}}{m}
 <1,
$$
where
$$
\frac{1}{m}:=\sup_{t\in [0,1]} \left\{ \int_{0}^{1}|k(t,s)|g(s)\,ds +
\frac{|\gamma(t)|}{1-\alpha[\gamma]}\int_{0}^{1}
\mathcal{K}_{A}(s)g(s)\,ds \right\}
$$
and 
$$
  F^{(-\rho,\rho)}:=\sup \left\{\frac{F(t,\phi)}{\rho }:\;
t\in [0,1], \;  \phi \in C([-r, 0], \R)
\;\text{ with } \; \|\phi\|_{[-r,0]} \le \rho \right\}.$$ 
\end{enumerate}{}
Then $i_{K_\psi}(\F,K_{\psi,\rho})=1$.
\end{lem}

\begin{proof}
We show that $\mu (u-\psi) \neq \F u-\psi$ for every $u \in \partial K_{\psi,\rho}$
and for every $\mu \geq 1$.

In fact, if this does not happen, there exist $\mu \geq 1$ and $u\in
\partial K_{\psi,\rho}$ such that $\mu (u-\psi)=\F u-\psi$,
that is
$$
\mu \big(u(t)-\psi(t)\big)=  \int_{0}^{1}
k(t,s)g(s)F(s,u_s)\,ds + \gamma(t) \alpha[u],
\quad \text{for every}\ t \in [-r,1].
$$
Applying $\alpha$ to both sides of the equation and noting that $\alpha[\psi]=0$, we get
$$\mu \alpha[ u] = \int_{0}^{1}
\mathcal{K}_{A}(s)g(s)F(s,u_s)\,ds + \alpha[\gamma]\alpha[u]$$
thus,  from $(C_7)$, $\mu-\a[\c]\geq 1-\a[\c]>0$, and we deduce that
$$\alpha[u]=\frac{1}{\mu-\alpha[\gamma]}\int_{0}^{1}
\mathcal{K}_{A}(s)g(s)F(s,u_s)\,ds$$
and we get, by substitution,
$$\mu \big(u(t)-\psi(t)\big)=  \int_{0}^{1}
k(t,s)g(s)F(s,u_s)\,ds + \frac{\gamma(t)}{\mu-\alpha[\gamma]}\int_{0}^{1}
\mathcal{K}_{A}(s)g(s)F(s,u_s)\,ds.$$
Recall that $u\in \partial K_{\psi,\rho}$ means that $u=\psi+v$ with $v \in \partial K_{0,\rho}$ and, in particular we have that $\|v\|_{[0,1]} =\rho$.
Now observe that $F(s,u_s) \leq \rho F^{(-\rho,\rho)}$ for all $s\in [0,1]$. This estimate follows
from the definition of $F^{(-\rho,\rho)}$ and the fact that, since $\|\psi\|_{[-r,0]}<\rho$, we have that $\|u_s\|_{[-r,0]} \leq \rho$ for all $s$.
Therefore, taking the absolute value and then the supremum for $t\in [-r,1]$ in the above equality, we get
\begin{align*}
\mu \rho \leq & 
\sup_{t\in [0,1]} \left\{ \int_{0}^{1}|k(t,s)|g(s)F(s,u_s)\,ds +
\frac{|\gamma(t)|}{\mu-\alpha[\gamma]}\int_{0}^{1}
\mathcal{K}_{A}(s)g(s)F(s,u_s)\,ds \right\}\\
\leq &\rho F^{(-\rho,\rho)}\cdot\sup_{t\in [0,1]}
\left\{ \int_{0}^{1}|k(t,s)|g(s)\,ds +
\frac{|\gamma(t)|}{1-\alpha[\gamma]}\int_{0}^{1}
\mathcal{K}_{A}(s)g(s)\,ds \right\}
 <\rho.
\end{align*}

This contradicts the fact that $\mu \geq 1$ and proves the result.
\end{proof}
\begin{rem}
If the condition $(\mathrm{I}_{\protect\rho }^{1})$ holds for a suitable $\rho> \|\psi\|_{[-r,0]}$, then the operator  $\F$ has a fixed point in $K_{\psi,\rho}$. Note that this fixed point could be a `trivial' solution (but with a nonzero norm) of the equation~\eqref{eqhamm}, namely
$$
u(t)=\begin{cases}
\psi(t),\ & t \in [-r,0], \\
0,\ & t \in [0,1].
\end{cases}
$$
\end{rem}
We now make use of the assumption $(C_{8})$ and we provide a condition that guarantees that the index is equal to zero on $V_{\psi,\rho}$, for some appropriate $\rho>0$.
\begin{lem}
\label{idx0b1} Assume that
\begin{enumerate}
\item[$(\mathrm{I}_{\protect\rho }^{0})$] there exist $\rho >0$ such that
such that
$$
\dfrac{F_{(\rho ,{\rho /c})}}{M(a,b)}>1,
$$
where
$$ 
\frac{1}{M(a,b)} :=\inf_{t\in [a,b]}\left\{\int_{a+r}^{b}k(t,s)g(s)\,ds +
\frac{\gamma(t)}{1-\alpha[\gamma]}\int_{a+r}^{b}
\mathcal{K}_{A}(s)g(s)\,ds \right\}$$
and
$$
F_{(\rho ,{\rho /c})} :=\inf \left\{\frac{F(t,\phi)}{\rho }%
:\; t\in [a,b], \;  \phi \in C([-r, 0], \R)
\;\text{ with } \; \phi(\theta) \in [\rho ,\rho /c]
\;\text{ for all } \; \theta \in [-r,0] \right\}.$$
\end{enumerate}

Then $i_{K_\psi}(\F,V_{\psi,\rho})=0$.
\end{lem}
\begin{proof}
Since $0\not\equiv\c\in K_0$ we can choose $e=\c$ in Lemma \ref{lemind}. We now prove that
\begin{equation*}
u\ne \F u+\sigma \gamma \quad\text{for  all } u\in \partial
V_{\psi,\rho}\text{ and } \sigma \geq 0.
\end{equation*}

In fact, if not, there exist $u\in \partial V_{\psi,\rho}$ and $\sigma\geq 0$ such that $u=\F u+\sigma \gamma$.
Then, in particular, we have
$$u(t)=\int_{0}^{1}
k(t,s)g(s)F(s,u_s)\,ds+\sigma \gamma(t) + \gamma(t) \alpha[u],\ \text{for every}\ t \in [0,1]$$
and
$$\alpha[ u] = \int_{0}^{1}
\mathcal{K}_{A}(s)g(s)F(s,u_s)\,ds + \sigma \alpha[\gamma] +\alpha[\gamma] \alpha[u].$$
Therefore we have,
$$\alpha[u]=\frac{1}{1-\alpha[\gamma]}\int_{0}^{1}
\mathcal{K}_{A}(s)g(s)F(s,u_s)\,ds+\frac{\sigma \alpha[\gamma]}{1-\alpha[\gamma]}$$
and, by substitution, we obtain
\begin{align*}
u(t)=&\int_{0}^{1}
k(t,s)g(s)F(s,u_s)\,ds+ \sigma \gamma(t)\\
&+\frac{\gamma(t)}{1-\alpha[\gamma]}\left(\int_{0}^{1}
\mathcal{K}_{A}(s)g(s)F(s,u_s)\,ds+\sigma\a[\gamma] \right).
\end{align*}

We claim that $F(s,u_s) \ge \rho F_{(\rho ,{\rho /c})}$ for all $s\in[a+r,b]$; observe that such an interval is nontrivial by $(C_{8})$.
In fact, since $u\in \partial V_{\psi,\rho}$, we have $u=\psi+v$ with $v \in \partial V_{0,\rho}$.
Consequently, given $s\in[a+r,b]$, we have that
$u_s(\theta) =  u(s+\theta)=v(s+\theta)$ for $\theta \in [-r,0]$ due to the facts that
$s+\theta\in[a,b] \subseteq (0,1]$ for all $\theta$ and that $\psi$ vanishes on $[0,1]$.
Furthermore, the function $v \in \partial V_{0,\rho}$ is such that 
$v(\tau) \in  [\rho ,\rho /c]$ for $\tau \in [a,b]$.
This follows from the definition of $V_{0,\rho}$ and from the inclusion
$V_{0,\rho}\subset K_{0,\rho/c}$.
Summing up we get that, if $s\in[a+r,b]$, then $u_s(\theta)=u(s+\theta) \in  [\rho ,\rho /c]$
for all $\theta \in [-r,0]$.
Therefore, by definition of the number $F_{(\rho ,{\rho /c})}$ we have $F(s,u_s) \ge \rho F_{(\rho ,{\rho /c})}$ for all $s\in[a+r,b]$, as claimed.

Hence we get, for $t\in[a,b]$,
\begin{align*}
u(t)&\ge
\int_{a+r}^{b} k(t,s)g(s)F(s,u_s)\,ds+\frac{\gamma(t)}{1-\alpha[\gamma]}
\int_{a+r}^{b} \mathcal{K}_{A}(s)g(s)F(s,u_s)\,ds
\\ &\ge
\rho F_{(\rho ,{\rho /c})}
\;  \left(
\int_{a+r}^{b} k(t,s)g(s)\,ds
+\frac{\gamma(t)}{1-\alpha[\gamma]}
\int_{a+r}^{b} \mathcal{K}_{A}(s)g(s)\,ds\right).
\end{align*}
Taking the minimum over $[a,b]$ gives
$\rho>\rho$, a contradiction.
\end{proof}

The above Lemmas can be combined in order to prove the following Theorem. Here we
deal with the existence of at least one, two or three nontrivial solutions.
We stress
that, by expanding the lists in conditions $(S_{5}),(S_{6})$ below, it is
possible to state results for four or more positive solutions, see for
example the paper by Lan~\cite{kljdeds} for the type of results that might be stated. We omit
the proof which follows directly from the properties of the fixed point index  stated in Lemma \ref{lemind}.
\begin{thm}
\label{thmmsol1} The integral equation \eqref{eqhamm} has at least one nontrivial solution
in $K_\psi$ if one of
the following conditions hold.
\begin{enumerate}
\item[$(S_{1})$] There exist $\rho _{1},\rho _{2}\in (0,\infty )$ with $\|\psi\|_{[-r,0]}<\rho _{2}$ and $\rho
_{1}/c<\rho _{2}$ such that $(\mathrm{I}_{\rho _{1}}^{0})$ and $(\mathrm{I}_{\rho _{2}}^{1})$ hold.
\item[$(S_{2})$] There exist $\rho _{1},\rho _{2}\in (0,\infty )$ with $\|\psi\|_{[-r,0]}<\rho
_{1}<\rho _{2}$ such that $(\mathrm{I}_{\rho _{1}}^{1})$ and $(\mathrm{I}%
_{\rho _{2}}^{0})$ hold.
\end{enumerate}
The integral equation \eqref{eqhamm} has at least two nontrivial solutions in $K_\psi$ if one of
the following conditions hold.
\begin{enumerate}
\item[$(S_{3})$] There exist $\rho _{1},\rho _{2},\rho _{3}\in (0,\infty )$
with $\|\psi\|_{[-r,0]}<\rho _{2}$ and $\rho _{1}/c<\rho _{2}<\rho _{3}$ such that $(\mathrm{I}_{\rho
_{1}}^{0}),$ $(
\mathrm{I}_{\rho _{2}}^{1})$ $\text{and}\;\;(\mathrm{I}_{\rho _{3}}^{0})$
hold.
\item[$(S_{4})$] There exist $\rho _{1},\rho _{2},\rho _{3}\in (0,\infty )$
with $\|\psi\|_{[-r,0]}<\rho _{1}<\rho _{2}$ and $\rho _{2}/c<\rho _{3}$ such that $(\mathrm{I}%
_{\rho _{1}}^{1}),\;\;(\mathrm{I}_{\rho _{2}}^{0})$ $\text{and}\;\;(\mathrm{I%
}_{\rho _{3}}^{1})$ hold.
\end{enumerate}
The integral equation \eqref{eqhamm} has at least three nontrivial solutions in $K_\psi$ if one
of the following conditions hold.
\begin{enumerate}
\item[$(S_{5})$] There exist $\rho _{1},\rho _{2},\rho _{3},\rho _{4}\in
(0,\infty )$ with $\|\psi\|_{[-r,0]}<\rho _{2}$ and $\rho _{1}/c<\rho _{2}<\rho _{3}$ and $\rho _{3}/c<\rho
_{4}$ such that $(\mathrm{I}_{\rho _{1}}^{0}),$ $(\mathrm{I}_{\rho _{2}}^{1}),\;\;(\mathrm{I}%
_{\rho _{3}}^{0})\;\;\text{and}\;\;(\mathrm{I}_{\rho _{4}}^{1})$ hold.
\item[$(S_{6})$] There exist $\rho _{1},\rho _{2},\rho _{3},\rho _{4}\in
(0,\infty )$ with $\|\psi\|_{[-r,0]}<\rho _{1}<\rho _{2}$ and $\rho _{2}/c<\rho _{3}<\rho _{4}$
such that $(\mathrm{I}_{\rho _{1}}^{1}),\;\;(\mathrm{I}_{\rho
_{2}}^{0}),\;\;(\mathrm{I}_{\rho _{3}}^{1})$ $\text{and}\;\;(\mathrm{I}%
_{\rho _{4}}^{0})$ hold.
\end{enumerate}
\end{thm}
\begin{rem}
Note that the solutions given by Theorem \ref{thmmsol1} are nontrivial in the sense that do not coincide with $\psi$; nevertheless, in view of
\eqref{norm-tra}, in the case of conditions $(S_{1})$, $(S_{3})$, $(S_{5})$ one of the solutions could have the same norm as $\psi$.
\end{rem}
\section{Non-negative solutions under stronger hypotheses} \label{sec-pos}
By means of  an  approach similar to that of the previous Section, we can prove the existence of solutions that are non-negative on $[0,1]$,
in the spirit of Remark 3.4 of \cite{gijwems} and Sections 2 and 3 of \cite{ac-gi-at-bvp}. To be more precise, we require that the maps $F$, $k$, $\psi$, $\gamma$, $\alpha$ and $g$ that occur in \eqref{eqhamm} and the delay $r$ satisfy the assumptions 
$(C_{1})-(C_{8})$ with $(C_{1})$, $(C_{3})$, $(C_{5})$ and $(C_{7})$ replaced with the following `positivity conditions'. 
\begin{enumerate}
\item [$(C'_{1})$] The function $\psi: [-r,1] \to [0,+\infty)$ is continuous and
such that $\psi(t)=0$ for all $t\in[0,1]$.
\item [$(C'_{3})$]
The kernel $k$ is non-negative in $[-r,1] \times [0,1]$ and there exist a subinterval $[a,b] \subseteq (0,1]$, a measurable function
$\Phi$ with $\Phi \geq 0$ a.\,e., and a constant $c_1=c_1(a,b) \in (0,1]$ such that
\begin{align*}
k(t,s)\leq \Phi(s) \text{ for  all }  &t \in [0,1] \text{ and a.\,e. } \, s\in [0,1],\\
k(t,s) \geq c_1\,\Phi(s) \text{ for  all } &t\in [a,b] \text{ and a.\,e. } \, s \in [0,1].
\end{align*}{}
\item [$(C'_{5})$] $F: [0,1] \times C([-r, 0], [0,\infty) )\to [0,\infty)$ is an operator that satisfies  
Carath\'eodory-type conditions as in $(C_{5})$. Furthermore,
for each $R>0$, there exists $\varphi_{R} \in
L^{\infty}[0,1]$ such that{}
\begin{equation*}
F(t,\phi) \le \varphi_{R}(t) \ \text{for all} \ \phi \in C([-r, 0], [0,\infty) )
\ \text{with} \ \|\phi\|_{[-r,0]} \le R,\ \text{and a.\,e.}\ t\in [0,1].
\end{equation*}{}
\item [$(C'_{7})$]
The function $\gamma: [-r,1] \to [0,\infty)$ is continuous, $\gamma\not\equiv 0$ and
such that $\gamma(t)=0$ for all $t\in[-r,0]$; moreover,
$0 \leq \alpha[\gamma] <1\;
\text{and there exists}\; c_{2} \in(0,1] \;\text{such that}\;
\gamma(t) \geq c_{2}\|\gamma\|_{[0,1]} \;\text{for all }\; t \in [a,b]$.
\end{enumerate}
Then, using the notation 
$$
P=\{u\in C([-r, 1], \R): u(t)\geq 0\ \text{for all}\ t\in[-r,1]\},
$$
it can be shown, by arguments similar to the previous Section, that $\F$ is compact and leaves the affine cone
\begin{equation}\label{poscone}
\mathcal K_{\psi}=\psi +( K_0 \cap P) 
\end{equation}
invariant. 
We now state two results analogous to Lemmas \ref{ind1b} and \ref{idx0b1}. The proofs, similar to the ones before, are omitted. Here, the sets $\mathcal K_{0,\rho}$ and $\mathcal V_{0,\rho}$ and the corresponding translates $\mathcal K_{\psi,\rho}$ and $\mathcal V_{\psi,\rho}$ are as in Definition \ref{translcone} with $K_0 \cap P$  in place of $K_0$.

\begin{lem}\label{idx1bpos}
Assume that 
\begin{enumerate}
\item[$(\overline{\mathrm{I}_{\protect\rho }^{1}})$] \label{EqBpos} there exists $\rho> \|\psi\|_{[-r,0]}$ such that
$$
\frac{F^{(0,\rho)}}{m}
 <1,
$$
where
$$
  F^{(0,\rho)}:=\sup \left\{\frac{F(t,\phi)}{\rho }:\;
t\in [0,1], \;  \phi \in C([-r, 0], [0,\infty))
\;\text{ with } \; \|\phi\|_{[-r,0]} \le \rho \right\}.$$ 
\end{enumerate}{}
Then $i_{\mathcal K_\psi}(\F, \mathcal K_{\psi,\rho})=1$.
\end{lem}
\begin{lem}
\label{idx0b1pos} Assume that
\begin{enumerate}
\item[$(\overline{\mathrm{I}_{\protect\rho }^{0}})$] there exist $\rho >0$ such that
such that
$$
\dfrac{F_{(\rho ,{\rho /c})'}}{M(a,b)}>1,
$$
where
$$
F_{(\rho ,{\rho /c})'} :=\inf \left\{\frac{F(t,\phi)}{\rho }%
:\; t\in [a,b], \;  \phi \in C([-r, 0], [0,+\infty))
\;\text{with} \; \phi(\theta) \in [\rho ,\rho /c]
\;\text{for all} \; \theta \in [-r,0] \right\}.$$
\end{enumerate}
Then $i_{\mathcal K_\psi}(\F,\mathcal V_{\psi,\rho})=0$.
\end{lem}
A result equivalent to Theorem~\ref{thmmsol1} is hold in this case, with nontrivial solutions belonging to the affine cone~\eqref{poscone}.
\section{Nontrivial solutions of some FBVP's}\label{ex-app}
In this Section we provide some applications of the results of Sections~\ref{sec-changesign} and \ref{sec-pos}. 
\subsection{Solutions that may change sign}
We illustrate the results of Section~\ref{sec-changesign}, by considering the FBVP
\begin{equation}\label{eq4.1}
-u''(t)=g(t) F(t,u_t),\ t \in [0,1],
\end{equation}
with initial conditions
\begin{equation}\label{eqic4.2}
u(t)=\psi(t),\ t \in [-r,0]
\end{equation}
and BCs
\begin{equation}\label{eqbc4.2}
u(0)=0,\; \beta u'(1) + u(\eta)=\a[u], \; {\beta}>0, \; {\eta}\in (0,1).
\end{equation}

The solution of the ODE $-u''=y$ under the BCs~\eqref{eqbc4.2} (a similar calculation, under a slightly different set of BCs, is done in~\cite{gijwems}) is given by
$$
u(t)= \frac{t}{\beta+\eta} \a [u] + \frac{\beta t}{\beta+\eta}\int_{0}^{1}y(s)ds
+\frac{t}{\beta+\eta}\int_{0}^{\eta}(\eta-s)y(s)ds
-\int_{0}^{t}(t-s)y(s)ds.
$$
By a solution of the FBVP \eqref{eq4.1}--\eqref{eqic4.2}--\eqref{eqbc4.2} we
mean a solution $u\in C[-r,1]$ of the corresponding integral
equation
$$
u(t)= \psi(t)
+\int_{0}^{1}k(t,s)g(s)F(s,u_s)ds
+\gamma(t)\a[u],\ t \in [-r,1],
$$
where $\gamma(t)= \dfrac{ t}{\beta+\eta}H(t)$ and
$$
k(t,s)=\left[\frac{\beta t}{\beta+\eta} +\dfrac{t}{\beta+\eta}(\eta-s)H(\eta-s)-
 (t-s) H(t-s)\right] H(t),
$$
with
$$
H(\tau)=
\begin{cases}
1,\ & \tau \geq 0, \\0,\ & \tau<0.
\end{cases}
$$
When $\beta\geq 0$, $k(t,s)$ changes sign when $0<\beta+\eta<1$, but is
non-negative on the strip $0\leq t \leq b$, $b<\beta+\eta$.
Assume $r<\beta+\eta$.
Then, we can apply the results of
Section~\ref{sec-changesign} to any interval $[a,b] \subset (0, \beta+\eta)$ of length $b-a>r$.
Observe that $(C_{7})$ holds with $c_2=a$.
We want to find $\Phi, c_1$ so that $(C_{3})$ holds.
For this purpose we follow the outline of \cite{gijwems} and take for simplicity
$$
\Phi (s)=\begin{cases}
s, & \text{ for }\beta+\eta\geq \frac{1}{2}, \\
\left[\dfrac{1 -(\beta + \eta )}{\beta+\eta}\right]s, & \text{ for
}\beta+\eta< \frac{1}{2}.
\end{cases}
$$
Then the upper bound $|k(t,s)|\leq \Phi(s)$ holds.
Concerning the lower bounds, we have that if
$\beta+\eta\geq \frac{1}{2}$, we may choose
$$
c_1=\min \Bigl\{ \frac{a \beta}{\beta+\eta},
\frac{\beta+\eta-b}{\beta+\eta} \Bigr\}.
$$
While if $\beta+\eta< \frac{1}{2}$, we may take
$$
c_1=\min \Bigl\{ \frac{a \beta}{1 -(\beta + \eta)},
\frac{\beta+\eta-b}{1 -(\beta + \eta)} \Bigr\}.
$$
Therefore we take
\begin{equation}\label{cbcA}
c=\begin{cases}
\min \Bigl\{ \dfrac{a \beta}{\beta+\eta},
\dfrac{\beta+\eta-b}{\beta+\eta} \Bigr\}, & \text{ for }\beta+\eta\geq \frac{1}{2}, \\
\min \Bigl\{ \dfrac{a \beta}{1 -(\beta + \eta)},
\dfrac{\beta+\eta-b}{1 -(\beta + \eta)} \Bigr\}
, & \text{ for
}\beta+\eta< \frac{1}{2}.
\end{cases}
\end{equation} 

Here we state, for brevity, a result regarding the existence of one nontrivial solution, that is a direct consequence of Theorem \ref{thmmsol1}. A similar result can be stated for the existence of multiple, nontrivial solutions.
\begin{thm} \label{example1}
Let $[a,b] \subset (0, \beta+\eta) \subset (0, 1)$ with $b-a>r$,
and let $c$ as in \eqref{cbcA} and $\int_a^b \Phi(s)g(s)\,ds >0$.
Then the FBVP \eqref{eq4.1}--\eqref{eqic4.2}--\eqref{eqbc4.2} has at least one nontrivial solution,
strictly positive on $[a,b]$, if either $(S_{1})$ or $(S_{2})$ of
Theorem~\ref{thmmsol1} holds.
\end{thm}

Now let $f:[0,1]\times\R\times\R \to [0,\infty)$ be a given Carath\'eodory map, and consider the following
\emph{delay differential equation}
\begin{equation}\label{eqdelay}
-u''(t)=g(t) f(t,u(t),u(t-r)),\ t \in [0,1],
\end{equation}
with $g$ non-negative and measurable.
The techniques developed in this paper can be applied to
study the nontrivial solutions of \eqref{eqdelay}
with BCs \eqref{eqic4.2}--\eqref{eqbc4.2}.
In fact, observe that the equation \eqref{eqdelay} is a special case of the
functional differential equation \eqref{eq4.1}.
In order to do this, given $f:[0,1]\times\R\times\R \to [0,\infty)$ as above, we proceed as in~\cite{halelunel} and
define $F: [0,1] \times C([-r, 0], \R) \to [0,\infty)$ by
$$
F(t,\phi)=f(t,\phi(0),\phi(-r)).
$$
Note that the operator $F$, defined in this way, verifies condition $(C_{5})$
provided that the map $f$ satisfies the following Carath\'eodory-type assumption:

$(C''_{5})$ For each $R>0$, there exists $\varphi^*_{R} \in
L^{\infty}[0,1]$ such that{}
\begin{equation*}
f(t,u,v) \le \varphi^*_{R}(t) \;\text{ for all } \; u,v \in \R
\;\text{ with } \; |u| \le R,\;|v| \le R,\;\text{ and a.\,e. } \; t\in [0,1].
\end{equation*}
In addition, in order to obtain from Theorem~\ref{thmmsol1} existence and multiplicity results
for the FBVP~\eqref{eqdelay}--\eqref{eqic4.2}--\eqref{eqbc4.2}, it is sufficient to consider
the following numbers:
\begin{align*}
  f^{(-\rho,\rho)}=&\sup \Bigl\{\frac{f(t,u,v)}{\rho }:\;
t\in [0,1], \;  |u|, |v| \le \rho \Bigr\},\\
f_{(\rho ,{\rho /c})} =&\inf \Bigl\{\frac{f(t,u,v)}{\rho }%
:\; t\in [a,b], \;  \rho \le u \le \rho /c, \; \rho \le v \le \rho /c \Bigr\}.
\end{align*}
These numbers are easier to compute than the analogous ones in the general case of a functional differential equation.

Let us consider, for illustrative purposes, the following autonomous equation depending on a positive parameter $\lambda$.
\begin{equation}\label{eqdelay1}
-u''(t)= \lambda \, |u(t)|^{p-1} |u(t-r)|,\ t \in [0,1],
\end{equation}
where $p \geq 1$,
with the initial conditions \eqref{eqic4.2}
and the boundary conditions
\begin{equation}\label{eqbc4.2.0}
u(0)=0,\; \frac14 u'(1) + u\big(\frac14\big)=0.
\end{equation}

Here we have $f(u,v)=|u|^{p-1} |v|$ so that $f^{(-\rho,\rho)} = \rho^{p-1} = f_{(\rho ,{\rho /c})}$.
Moreover we have
$$k(t,s)=\left[\frac12 t + 2t\left(\frac14-s\right)H\left(\frac14-s\right)-
 (t-s) H(t-s)\right] H(t).$$
A direct calculation shows that
$$\sup_{t\in [0,1]}  \int_{0}^{1}|k(t,s)|\,ds =\frac{17}{16}$$
and therefore, in this case, $m=16/17$.

As a consequence of 
Theorem \ref{example1} (using $(S_{2})$ of Theorem \ref{thmmsol1})  
we get the following.
\begin{cor}% \label{example1cor}
Let $[a,b]=[1/4,7/16]$, and let $c_2=1/4$ and $c_1=1/8$.
Assume that $r<3/16$ and let  $\psi$ with $\|\psi\|_{[-r,0]}<1$ be given.
Then, for every $0<\lambda < 16/17$
the FBVP~\eqref{eqdelay1}--\eqref{eqic4.2}--\eqref{eqbc4.2.0} has at least one nontrivial solution
$u_\lambda$, strictly positive on $[1/4,7/16]$, with $\|u_\lambda\|_{[-r,1]}>1$.
\end{cor}

\begin{proof}
Take $\rho_1=1$ and observe that
$(\mathrm{I}_{\rho _{1}}^{1})$ holds since $\lambda < 16/17=m$.
Moreover, for $\rho_2$ large enough (precisely $\rho_2 > \dfrac{M(a,b)}\lambda$)
condition $(\mathrm{I}_{\rho _{2}}^{0})$ holds as well.
Thus, Theorem \ref{thmmsol1} $(S_{2})$ can be applied,
yielding at least one solution
$u_\lambda$, positive on $[1/4,7/16]$, with $1<\|u_\lambda\|_{[-r,1]}<\rho_2$.
\end{proof}

\subsection{Non-negative solutions}
We now show the applicability of the tools of Section \ref{sec-pos}.

Firstly we consider the FBVP \eqref{eq4.1}--\eqref{eqic4.2}--\eqref{eqbc4.2} 
with a non-negative initial datum $\psi$, and look for non-negative solutions assuming  
$\beta+\eta\geq 1$.
For a fixed value $r<1$ of the delay, we can apply the results of
Section~\ref{sec-pos} to an arbitrary interval $[a,b] \subset (0, 1)$ of length $b-a>r$.
Note that when $\beta+\eta\geq 1$ the kernel $k$ is non-negative and, reasoning as in \cite{jw-na05}, we take 
$$
\Phi(s)=\begin{cases}
\dfrac{\beta}{\beta+\eta}s,&\text{ if } s \geq \eta,\\
s\Bigl(1-\dfrac{s}{\beta+\eta}\Bigr), &\text{ if } s < \eta,
\end{cases}
$$
so that the upper bound $k(t,s)\leq \Phi(s)$ holds.
For the lower bounds, a
routine calculation shows that $k(t,s) \geq c_1 \Phi(s)$ for
$t \in [a,b], s\in [0,1]$ if
$$
c_1= \min\Bigl\{a,1-\dfrac{b}{\beta+\eta}\Bigr\}.
$$
Moreover, $(C'_{7})$ holds with $c_2=a$.
Thus we work in the affine cone~\eqref{poscone} with
$c=c_1$, obtaining
an existence result for non-negative solutions analogous to Theorem~\ref{example1}.
\medskip

Finally we turn our attention to the FBVP
\begin{equation}\label{kmt-eq-app}
u''(t)+F(t,u_t)=0,\ t \in [0,1],
\end{equation}
with initial conditions
\begin{equation}\label{kmt-i-app}
u(t)=\psi(t),\ t \in [-r,0],
\end{equation}
and boundary conditions (BCs)
\begin{equation}\label{kmt-bcs-app}
u(0)=0,\;  u(1)=\a[u],
\end{equation}
where $\psi$ is non-negative. This FBVP can be seen as a generalization of the FBVP~\eqref{kmt-eq}--\eqref{kmt-i}--\eqref{kmt-bcs}, since the BCs involve a more general functional given by a signed measure.
To the FBVP~\eqref{kmt-eq-app}--\eqref{kmt-i-app}--\eqref{kmt-bcs-app} we associate the perturbed integral equation 
$$
u(t)= \psi(t)
+\int_{0}^{1}k(t,s)g(s)F(s,u_s)ds
+\gamma(t)\a[u],\ t \in [-r,1],
$$
where 
$$\gamma(t)= t H(t)\quad \text{and}\quad
k(t,s)=\left[t(1-s) -
 (t-s) H(t-s)\right] H(t).
$$
Clearly $k$ and $\gamma$ are non-negative.
In a similar way as in \cite{jwgi-nodea-08}, we may choose
$$
\Phi (s)=s(1-s)
$$
Then we have $k(t,s)\leq \Phi(s)$. Furthermore, for a fixed $[a,b]\subset (0,1)$,  by direct calculation we obtain 
$$
c_1=\min \{ a, 1-b  \}\ \text{and}\ c_2=a.
$$
Thus we may take
\begin{equation}\label{lastc}
c=\min\{ a, 1-b  \},
\end{equation} 
and work in the affine cone~\eqref{poscone} with $c$ given by~\eqref{lastc}.
An analogue of Theorem~\ref{example1} holds in this case as well.

\section*{Acknowledgments}
The authors would like to thank the anonymous Referee for the careful reading of the manuscript and for the constructive comments. This paper was partially written during the visit of G. Infante to the Dipartimento di Ingegneria Industriale e Scienze Matematiche of the Universit\`{a} Politecnica delle Marche. G. Infante is grateful to the people of the aforementioned Dipartimento for their kind and warm hospitality. The authors were partially supported by G.N.A.M.P.A. - INdAM (Italy).

\end{document}